\documentclass[a4paper,reqno]{amsart}

\usepackage{textcase, cmap, amsmath, amsfonts, amsthm, bbm, amssymb, mathtools, indentfirst, multirow, multicol, enumerate, tasks, tikz, esint, graphicx}

\usetikzlibrary{calc}
\usepackage[many]{tcolorbox}
\usepackage[unicode]{hyperref}

\PassOptionsToPackage{alphabetic}{amsrefs}
\usepackage{amssymb, amsthm, amsfonts, amsmath, mathrsfs, latexsym, pstricks-add, amsrefs, graphicx, tikz, thmtools, enumitem}
\usepackage[mathscr]{euscript}
\usepackage{lmodern}
\usepackage[utf8]{inputenc}
\usepackage[T1]{fontenc}
\usepackage{microtype}

\definecolor{Maroon}{RGB}{140,10,0}

\hypersetup{
	colorlinks,
	linkcolor={Maroon},
	citecolor={Maroon},
	urlcolor={Maroon}
}

\newtheorem{theorem}{Theorem}
\newtheorem{lemma}[theorem]{Lemma}
\newtheorem{corollary}[theorem]{Corollary}

\theoremstyle{definition}

\newtheorem{remark}[theorem]{Remark}

\newcommand{\di}{\displaystyle}

\newcommand{\N}{\mathbb N}
\newcommand{\Z}{\mathbb Z}

\newcommand{\R}{\mathbb R}

\let\ge\geqslant
\let\geq\geqslant
\let\le\leqslant
\let\leq\leqslant
\let\epsilon\varepsilon

\graphicspath{{./figs/}}

\begin{document}
	\title[Mean value theorem and Ostrowski inequality for $q$-calculus]{Mean value theorem for quantum integral operator with application to sharp Ostrowski inequality}
	
	\author[A.~Aglić~Aljinović]{Andrea Aglić Aljinović}
	\address{University of Zagreb, Faculty of Electrical Engineering and Computing, Unska 3, 10 000 Zagreb, Croatia}
	\email{andrea.aglic@fer.hr}
	
	\author[D.~Kovačević]{Domagoj Kovačević}
	\address{University of Zagreb, Faculty of Electrical Engineering and Computing, Unska 3, 10 000 Zagreb, Croatia}
	\email{domagoj.kovacevic@fer.hr}
	
	\author[M.~Puljiz]{Mate Puljiz}
	\address{University of Zagreb, Faculty of Electrical Engineering and Computing, Unska 3, 10 000 Zagreb, Croatia}
	\email{mate.puljiz@fer.hr}
	
	\author[A.~Žgaljić~Keko]{Ana Žgaljić Keko}
	\address{University of Zagreb, Faculty of Electrical Engineering and Computing, Unska 3, 10 000 Zagreb, Croatia}
	\email{ana.zgaljic@fer.hr}
	
	\date{\today}
	
	\begin{abstract}
		We derive a version of Lagrange's mean value theorem for quantum calculus.
		We disprove a version of Ostrowski inequality for quantum calculus appearing in the literature.
		We derive a correct statement and prove that our new inequality is sharp.
		We also derive a midpoint inequality.
	\end{abstract}

	\maketitle
	
	\section{Introduction}
	Quantum calculus is calculus based on finite difference principle without the concept of limits. It has two main branches $q$-calculus and $h$-calculus. While $h$-calculus has important applications for ordinary and partial differential equations, optimization theory and in applied fields such as physics, engineering and economics (see \cites{AFC,FOP,NLFDE}), $q$-calculus has applications in number theory, combinatorics, fractals, approximation theory, numerical analysis, ordinary and partial difference equations, dynamical systems, quantum groups, quantum algebras, Lie algebras, complex analysis, computer science, particle physics and quantum mechanics (see \cites{IIQC,FitouhiBrahim,IMCA,qCalcComprehensive,HighOrderSingularBVP,NoorAwan,BPST,MezliniOuled,AhasanMursaleen}).

	\medskip
	
	The well known \textbf{Ostrowski inequality} gives an estimate of the difference of
	function values and its integral mean on a segment
	\cite{OSTR}:%
	\begin{equation}
	\left\vert f(x)-\frac{1}{b-a}\int_{a}^{b}f(t)dt\right\vert \leq\left[
	\frac{1}{4}+\frac{\left(  x-\frac{a+b}{2}\right)  ^{2}}{\left(  b-a\right)
		^{2}}\right]  \left(  b-a\right)  \left\Vert f^{\prime}\right\Vert _{\infty}.
	\label{OST}%
	\end{equation}
	It holds for every $x\in\left[  a,b\right]  $ whenever $f:\left[  a,b\right]
	\rightarrow%
	%TCIMACRO{\U{211d} }%
	%BeginExpansion
	\mathbb{R}
	%EndExpansion
	$\ is continuous on $\left[  a,b\right]  $\ and differentiable on $\left(
	a,b\right)  $\ with derivative $f^{\prime}:\left(  a,b\right)  \rightarrow%
	%TCIMACRO{\U{211d} }%
	%BeginExpansion
	\mathbb{R}
	%EndExpansion
	$\ bounded on $\left(  a,b\right)  $\ i.e.%
	\[
	\left\Vert f^{\prime}\right\Vert _{\infty}:=\sup_{t\in\left(  a,b\right)
	}\left\vert f^{\prime}\left(  t\right)  \right\vert <+\infty.
	\]
	
	In a published paper \cite{TARI2}, the following version of Ostrowski inequality was obtained for quantum calculus.
	
	\begin{theorem}[Incorrect result {\cite{TARI2}*{Theorem 3.5}}]
		Let $f:\left[  a,b\right]  \rightarrow%
		%TCIMACRO{\U{211d} }%
		%BeginExpansion
		\mathbb{R}
		%EndExpansion
		$ be a $q$-differentiable function with $D_{q}^{a}f$ continuous on $[a,b]$ and
		$0<q<1$. Then we have%
		\begin{align}
		&  \left\vert f(x)-\frac{1}{b-a}%
		%TCIMACRO{\dint \limits_{a}^{b}}%
		%BeginExpansion
		{\displaystyle\int\limits_{a}^{b}}
		%EndExpansion
		f\left(  t\right)  d_{q}^{a}t\right\vert \nonumber\\
		&  \leq\left[  \frac{2q}{1+q}\left(  \frac{x-\frac{\left(  3q-1\right)
				a+\left(  1+q\right)  b}{4q}}{b-a}\right)  ^{2}+\frac{-q^{2}+6q-1}{8q\left(
			1+q\right)  }\right]  \left(  b-a\right)  \left\Vert D_{q}^{a}f\right\Vert
		_{\infty}. \label{OST2}%
		\end{align}
		
	\end{theorem}
	
	As we shall see later, this inequality does hold for some $x\in[a,b]$ but not all as claimed.
	The proof given in \cite{TARI2} uses the standard Lagrange's mean value theorem, which, as we show in Section \ref{sec:MVT}, does not hold in quantum calculus. In Section \ref{sec:OstrowskiOnLattice} we give a counterexample to \eqref{OST2} and in Section \ref{sec:OstrowskiFull} we derive a correct version of this inequality that holds for all $x\in[a,b]$.
	
	\medskip
	
	The rest of the paper is organised as follows. In Section \ref{sec:prelim} we give preliminaries for quantum calculus. In Section \ref{sec:MVT} we obtain Lagrange's mean value theorem for quantum calculus and use it in Section \ref{sec:OstrowskiOnLattice} to
	obtain a sharp bound for $q$-Ostrowski inequality for $x=a+q^{m}\left(
	b-a\right)  $, $m\in%
	%TCIMACRO{\U{2115} }%
	%BeginExpansion
	\mathbb{N}
	%EndExpansion
	\cup\left\{  0\right\}  $. Finally, in Section \ref{sec:OstrowskiFull}, we obtain Ostrowski inequality
	for all possible values $x\in [a,b]$ and show that our bound is	optimal.
	
	\section{\texorpdfstring{$q$}{q}-calculus preliminaries}\label{sec:prelim}
	
	F. H. Jackson \cite{JACK} in 1908 has defined what is now known as \textbf{Euler-Jackson
		$q$-difference operator} ($q$-derivative of the function) by%
	$$
	D_{q}f\left(  x\right)   =\frac{f\left(  x\right)  -f\left(  qx\right)
	}{\left(  1-q\right)  x},\text{ }x\in\left\langle 0,b\right]  ,\text{\ }%
	q\in\left\langle 0,1\right\rangle
	$$
	for an arbitrary function $f:\left[  0,b\right]  \rightarrow%
	%TCIMACRO{\U{211d} }%
	%BeginExpansion
	\mathbb{R}
	%EndExpansion
	$, where $b>0$. Note that every such function is
	$q$-differentiable for every $x\in\left\langle 0,b\right]$. When $\underset{x\rightarrow0}{\lim}D_{q}f\left(  x\right)$ exists it is said that $f$ is $q$-differentiable on $\left[  0,b\right]  $ and 
	$$
	D_{q}f\left(  0\right)   =\underset{x\rightarrow0}{\lim}D_{q}f\left(x\right).
	$$
	
	The $q$-derivative is a
	discretization of ordinary derivative and if $f$ is differentiable function
	then
	\[
	\underset{q\rightarrow1}{\lim}D_{q}f\left(  x\right)  =f^{\prime}\left(
	x\right)  .
	\]

	F. H. Jackson \cite{JACK2} in 1910 has also defined \textbf{$q$-integral} (or
	Jackson integral) by
	\begin{equation*}%
	%TCIMACRO{\dint \limits_{0}^{x}}%
	%BeginExpansion
	{\displaystyle\int\limits_{0}^{x}}
	%EndExpansion
	f\left(  t\right)  d_{q}t=\left(  1-q\right)  x%
	%TCIMACRO{\dsum \limits_{k=0}^{\infty}}%
	%BeginExpansion
	{\displaystyle\sum\limits_{k=0}^{\infty}}
	%EndExpansion
	q^{k}f\left(  q^{k}x\right)  ,\text{ \ }x\in\left\langle 0,b\right].
	\end{equation*}
	If the series on the right hand side is convergent, then $q$-integral $%
	%TCIMACRO{\dint _{0}^{x}}%
	%BeginExpansion
	{\displaystyle\int_{0}^{x}}
	%EndExpansion
	f\left(  t\right)  d_{q}t$ exists. If $f$ is continuous on $\left[
	0,b\right]  $ as $q\rightarrow1$ the series $\left(  1-q\right)  x%
	%TCIMACRO{\dsum \limits_{k=0}^{\infty}}%
	%BeginExpansion
	{\displaystyle\sum\limits_{k=0}^{\infty}}
	%EndExpansion
	q^{k}f\left(  q^{k}x\right)  $ tends to the Riemann integral (\cite{ANNA},
	\cite{KAC})
	\[
	\underset{q\rightarrow1}{\lim}%
	%TCIMACRO{\dint \limits_{0}^{x}}%
	%BeginExpansion
	{\displaystyle\int\limits_{0}^{x}}
	%EndExpansion
	f\left(  t\right)  d_{q}t=%
	%TCIMACRO{\dint \limits_{0}^{x}}%
	%BeginExpansion
	{\displaystyle\int\limits_{0}^{x}}
	%EndExpansion
	f\left(  t\right)  dt
	\]

	Previous definitions and results for $f:\left[  0,b\right]  \rightarrow%
	%TCIMACRO{\U{211d} }%
	%BeginExpansion
	\mathbb{R}
	%EndExpansion
	$ can easily be generalized for $f:\left[  a,b\right]  \rightarrow%
	%TCIMACRO{\U{211d} }%
	%BeginExpansion
	\mathbb{R}
	%EndExpansion
	$ (see \cite{TARI}). 
	If we have a function $f:\left[  a,b\right]  \rightarrow%
	%TCIMACRO{\U{211d} }%
	%BeginExpansion
	\mathbb{R}
	%EndExpansion
	$ then ``shifted'' $q$-derivative for $q\in\left\langle 0,1\right\rangle $ can be
	defined as
	$$
	D_{q}^{a}f\left(  x\right)  =\frac{f\left(  x\right)  -f\left(  a+q\left(
		x-a\right)  \right)  }{\left(  1-q\right)  \left(  x-a\right)  }, \quad \text{ if
	}x\in\left\langle a,b\right] ,
	$$
	If $\underset{x\rightarrow a}{\lim}D_{q}^{a}f\left(  x\right)  $ exists,
	$f:\left[  a,b\right]  \rightarrow \mathbb{R}$ is said to be $q$-differentiable on $[a,b]$ and
	$$D_{q}^{a}f\left(  a\right)  =\underset{x\rightarrow a}{\lim}D_{q}^{a}f\left(  x\right).
	$$
	``Shifted'' $q$-integral is defined by
	\[%
	%TCIMACRO{\dint \limits_{a}^{x}}%
	%BeginExpansion
	{\displaystyle\int\limits_{a}^{x}}
	%EndExpansion
	f\left(  t\right)  d_{q}^{a}t=\left(  1-q\right)  \left(  x-a\right)
	%TCIMACRO{\dsum \limits_{k=0}^{\infty}}%
	%BeginExpansion
	{\displaystyle\sum\limits_{k=0}^{\infty}}
	%EndExpansion
	q^{k}f\left(  a+q^{k}\left(  x-a\right)  \right)  ,\text{ \ }x\in\left[
	a,b\right]  .
	\]
	If the series on the right hand side is convergent, then $q$-integral $%
	%TCIMACRO{\dint _{a}^{x}}%
	%BeginExpansion
	{\displaystyle\int_{a}^{x}}
	%EndExpansion
	f\left(  t\right)  d_{q}^{a}t$ exists and $f:\left[  a,b\right]  \rightarrow%
	%TCIMACRO{\U{211d} }%
	%BeginExpansion
	\mathbb{R}
	%EndExpansion
	$ is $q$-integrable on $\left[  a,x\right]  $. If $c\in\left\langle
	a,x\right\rangle $ $q$-integral over $[c,x]$ is defined by
	\[%
	%TCIMACRO{\dint \limits_{c}^{x}}%
	%BeginExpansion
	{\displaystyle\int\limits_{c}^{x}}
	%EndExpansion
	f\left(  t\right)  d_{q}^{a}t=%
	%TCIMACRO{\dint \limits_{a}^{x}}%
	%BeginExpansion
	{\displaystyle\int\limits_{a}^{x}}
	%EndExpansion
	f\left(  t\right)  d_{q}^{a}t-%
	%TCIMACRO{\dint \limits_{a}^{c}}%
	%BeginExpansion
	{\displaystyle\int\limits_{a}^{c}}
	%EndExpansion
	f\left(  t\right)  d_{q}^{a}t.
	\]

%	If $f(x)$ and $g(x)$ are two real functions defined on $\left[  a,b\right]  $
%	whose ordinary derivatives exist in a neighbourhood of $x=a$ and are continuous
%	at $x=a$, then we have $q$-integration by parts formula%
%	
%	\[%
%	%TCIMACRO{\dint \limits_{a}^{b}}%
%	%BeginExpansion
%	{\displaystyle\int\limits_{a}^{b}}
%	%EndExpansion
%	f\left(  t\right)  D_{q}^{a}g\left(  t\right)  d_{q}^{a}t=f\left(  b\right)
%	g\left(  b\right)  -f\left(  a\right)  g\left(  a\right)  -%
%	%TCIMACRO{\dint \limits_{a}^{b}}%
%	%BeginExpansion
%	{\displaystyle\int\limits_{a}^{b}}
%	%EndExpansion
%	g\left(  qt+\left(  1-q\right)  a\right)  D_{q}^{a}f\left(  t\right)
%	d_{q}^{a}t.
%	\]
	\medskip
	
	An important difference between the definite $q$-integral and Riemann integral
	is that even if we are integrating a function on an interval $[c,b]$, $a<c<b$
	we have to take into account its behaviour at $t=a$ as well as its values on
	$[a,c]$. Beside the improper use of Lagrange's mean value theorem, this is the other reason
	for mistakes made in \cite{TARI2}.
	
	\begin{remark}
		In case $a=0$, when writing $D_q^0 f$ and $\di\int d_q^0 t$, we shall omit superscript zeros. This is consistent with the notation for original Jackson derivative and integral.
	\end{remark}
	
	\section{Lagrange's mean value theorem for \texorpdfstring{$q$}{q}-calculus}\label{sec:MVT}
	
	Here and hereafter the symbol $\left\Vert \cdot\right\Vert _{\infty}^{\langle a,b]}$ denotes the supremum
	\[
	\left\Vert f\right\Vert _{\infty}^{\langle a,b]}=\sup
	_{t\in\left\langle a,b\right]  }\left\vert f\left(  t\right)  \right\vert .
	\]
	\begin{remark}
		This is not a norm on the space of all functions with domain $[a,b]$, however, it is a norm, and it coincides with the standard $\|\cdot\|_\infty$ norm, for the class of functions that are continuous at $a$.
	\end{remark}

	\begin{remark}
		Recall that every function $f\colon [a,b]\to\R$ has a $q$-derivative $D_q^a f(t)$ for any $t\in\langle a, b]$. Further, $D_q^a f(a) = \lim_{x\to a} D_q^a f(x)$ when this limit exists. Thus, $q$-differentiable functions are, by definition, continuously differentiable at $a$, and, therefore, the norm of the derivative $\|D_q^a f \|_\infty = \sup_{t\in [a,b]} |f(t)|$ is the same as $\|D_q^a f\|_\infty^{\langle a,b]}$.
		
		Some of the results that follow do not require $q$-differentiability of $f$ (at $x=a$) and this notation allows us to state them in full generality.
	\end{remark}

	\medskip
	
	Let us first see that the standard mean value theorem does not hold in $q$-calculus. Setting $a=0$ and $b=2$, consider a function
	\[
	f\left(  x\right)  =\left\{
	\begin{array}
	[c]{cc}%
	1, & x\in\left[  1,2\right], \\
	0, & x\in\left[  0,1\right\rangle.
	\end{array}
	\right.
	\]
	Clearly $\left\Vert D_{q}f\right\Vert _{\infty}=\frac{1}{1-q}$, but
	\begin{equation}
	\left\vert f\left(  x\right)  -f\left(  y\right)  \right\vert \not\leq\left\Vert
	D_{q}^{a}f\right\Vert _{\infty}\left\vert
	x-y\right\vert \label{La}
	\end{equation}
	for $x=1$ and any $y\in\left\langle 0,q\right\rangle$.
	
	Note that this function, although not continuous, is $q$-differentiable. It is not hard to find examples of continuous functions which also fail the standard mean value theorem. For instance, one could take the same example we give in Remark \ref{rem:protuprimjer} below.
	
	\medskip
	
	In the next theorem we show that inequality \eqref{La} does hold when both $x$ and $y$ belong to the same \emph{$q$-lattice}, which is to say that $y=a+q^{n}\left(x-a\right)$, for some $n\in\Z$.
	
	\begin{theorem}[Mean value inequality for $q$-calculus]\label{thm:MVT}
		Let $f \colon [a,b] \to \R$ be an arbitrary function. Then for all $x\in\left\langle a,b\right]$, and all $n\in\N\cup\{0\}$ we have
		\begin{equation*}
		\left\vert f\left(  x\right)  -f\left(  a+q^{n}\left(  x-a\right)
			\right)\right\vert
		\leq \left| x-\left(  a+q^{n}\left(  x-a\right)  \right)\right|\, \left\Vert D_{q}^{a}f\right\Vert _{\infty}^{\left\langle a,b\right]}.
		\end{equation*}
	\end{theorem}
		
		\begin{proof} For $n=0$ the statement is trivial. When $n>0$, one can write
			\begin{align*}
				\frac{f(x) - f(a+q^{n} (x-a))}{(1-q)(x-a)} &=  \sum_{i=0}^{n-1}\frac{ f(a+q^{i}(x-a)) - f(a+q^{i+1} (x-a)) }{(1-q)(x-a)}\\
				&=  \sum_{i=0}^{n-1} q^i \cdot \frac{ f(a+q^{i}(x-a)) - f(a+q^{i+1} (x-a)) }{(q^i-q^{i+1})(x-a)}\\
				&=  \sum_{i=0}^{n-1} q^i \cdot D_q^a f(a+q^{i}(x-a)).
			\end{align*}
			Dividing this by $1+q+q^{2}+\cdots+q^{n-1} = \frac{1-q^n}{1-q}$ we obtain
			\begin{equation*}
			\frac{f(x) - f(a+q^{n} (x-a))}{(1-q^n)(x-a)} = \frac{\sum_{i=0}^{n-1} q^i \cdot D_q^a f(a+q^{i}(x-a))}{\sum_{i=0}^{n-1} q^i}.
			\end{equation*}
			On the right hand side, we have a weighted average of numbers
			$$D_q^a f(x),\, D_q^a f(a+q(x-a)),\, \dots ,\, D_q^a f(a+q^{n-1}(x-a))$$
			with weights $1, q, q^2, \dots q^{n-1}$ respectively. This average must therefore be in-between $\di\min_{0\le i < n} D_q^a f(a+q^i (x-a))$ and $\di\max_{0\le i < n} D_q^a f(a+q^i (x-a))$ and in particular
			\begin{equation*}
			\left|\frac{f(x) - f(a+q^{n} (x-a))}{(1-q^n)(x-a)} \right | \le \| D_q^a f\|_\infty^{\langle a,b]}.
			\end{equation*}
			Multiplying both sides by $\left|(1-q^n)(x-a)\right|$ we obtain the claim.
		\end{proof}
	
	\begin{corollary}\label{cor:MVT}
		Let $f \colon [a,b] \to \R$ be an arbitrary function. Then
		$$
		\left\vert f\left(a+q^{m}\left(  b-a\right) \right)  -f\left(  a+q^{k}\left(b-a\right)
		\right)\right\vert
		\leq  (b-a)\, |q^m-q^k| \, \left\Vert D_{q}^{a}f\right\Vert _{\infty}^{\left\langle a,b\right]}.
		$$
		holds for all $m,k\in\N\cup\{0\}$.
	\end{corollary}
	\begin{proof}
		Without loss of generality we can assume $k\ge m$. If we now set \linebreak $x=a+q^m (b-a)$ and $n=k-m$, and then apply Theorem \ref{thm:MVT}, the statement follows.
	\end{proof}

	If we additionally assume that the function in the Theorem \ref{thm:MVT} is continuous, we can obtain a $q$-calculus version of Lagrange's mean value theorem.
	\begin{theorem}[Lagrange's mean value theorem for $q$-calculus]
		Let $f \colon [a,b] \to \R$ be a continuous function and let $x,y\in\left\langle a,b\right]$ be such that $y=a+q^n(x-a)$ for some $n\in\N\cup\{0\}$. Then there exists some $c\in[ y,x ]$ such that
		\begin{equation*}
		f(x)-f(y) = D_q^a f(c)(x-y).
		\end{equation*}
	\end{theorem}
	\begin{proof}
		For $x=y$ ($n=0$) the statement is trivial, so assume $x\neq y$ ($n>0$). From the proof of Theorem \ref{thm:MVT} we have
		\begin{equation*}
		\min_{0\le i < n} D_q^a f(a+q^i (x-a)) \le \frac{f(x) - f(y)}{x-y} \le \max_{0\le i < n} D_q^a f(a+q^i (x-a)).
		\end{equation*}
		
		The function $f$ is continuous on $[a,b]$, hence, its $q$-derivative
		$$D_q^a f (t)= \frac{f(t)-f(a+q(t-a))}{(1-q)(t-a)}$$
		is also continuous on $\langle a,b]$ and in particular on $[y,x]$. Since
		$$\min_{t\in[y,x]}D_q^a f(t) \le \min_{0\le i < n} D_q^a f(a+q^i (x-a)) \le \max_{0\le i < n} D_q^a f(a+q^i (x-a)) \le \max_{t\in[y,x]} D_q^a f(t) $$
		there must exist some $c\in[y,x]$ such that
		$$\frac{f(x) - f(y)}{x-y} = D_q^a f(c),$$
		which proves the claim.
	\end{proof}

	\section{\texorpdfstring{$q$}{q}-Ostrowski inequality for points on \texorpdfstring{$q$}{q}-lattice}\label{sec:OstrowskiOnLattice}
	
	In the following theorem we give a correct proof of $q$-Ostrowski inequality for the points on the $q$-lattice of the form $x=a+q^m(b-a)$. For these values of $x$, the estimate we obtain here is matching that from \cite{TARI2}, but the proof given there is incorrect.
	
	\begin{theorem}[$q$-Ostrowski inequality on $q$-lattice]\label{tm:OSTqNet}
		Let $f\colon [a,b]\to\R$ be a $q$-integrable function over $\left[  a,b\right]  $. Then for every $m\in%
		%TCIMACRO{\U{2115} }%
		%BeginExpansion
		\mathbb{N}
		%EndExpansion
		\cup\left\{  0\right\}  $ the following inequality holds%
		\begin{equation}
		\left\vert f\left(  a+q^{m}\left(  b-a\right)  \right)  -\frac{1}{b-a}%
		%TCIMACRO{\dint \limits_{a}^{b}}%
		%BeginExpansion
		{\displaystyle\int\limits_{a}^{b}}
		%EndExpansion
		f\left(  t\right)  d_{q}^{a}t\right\vert \leq\left(  b-a\right)  \left(
		\frac{1+2q^{2m+1}}{1+q}-q^{m}\right)  \left\Vert D_{q}^{a}f\right\Vert
		_{\infty}^{\left\langle a,b\right]  }. \label{qOST}%
		\end{equation}
	\end{theorem}
	
	\begin{proof} Let $k\in\N\cup\{ 0 \}$ be arbitrary. From Corollary \ref{cor:MVT} we have
		$$
		\left\vert f\left(a+q^{m}\left(  b-a\right) \right)  -f\left(  a+q^{k}\left(b-a\right)
		\right)\right\vert
		\leq\left\Vert D_{q}^{a}f\right\Vert _{\infty}^{\left\langle a,x\right]} \, |q^m-q^k|\cdot (b-a).
		$$
		Note that
		\begin{multline*}
		f(a+q^{m}(b-a)) -\frac{1}{b-a}\int\limits_{a}^{b}
		f(t)\, d_{q}^{a}t =\\ = \sum\limits_{k=0}^\infty \left[ f(a+q^m(b-a)) - f(a+q^k(b-a)) \right](1-q)q^k
		\end{multline*}
		and therefore
		\begin{multline*}
			\left|f(a+q^{m}(b-a)) - \frac{1}{b-a}\int\limits_{a}^{b}
			f(t)\, d_{q}^{a}t\right| \le \\
			\begin{aligned}
			&\le (1-q) \sum\limits_{k=0}^\infty \left| f(a+q^m(b-a)) - f(a+q^k(b-a)) \right|q^k \\
			&\le (1-q)(b-a) \| D_q^a f \|_\infty^{\langle a, b]} \sum\limits_{k=0}^\infty \left|q^m - q^k\right|q^k.
			\end{aligned}
		\end{multline*}
		Finally, since
		\begin{align}%
		%TCIMACRO{\dsum \limits_{k=0}^{\infty}}%
		%BeginExpansion
		\nonumber{\displaystyle\sum\limits_{k=0}^{\infty}}
		%EndExpansion
		q^{k}\left\vert q^{m}-q^{k}\right\vert  &  =%
		%TCIMACRO{\dsum \limits_{k=0}^{m-1}}%
		%BeginExpansion
		{\displaystyle\sum\limits_{k=0}^{m-1}}
		%EndExpansion
		q^{k}\left(  q^{k}-q^{m}\right)  +%
		%TCIMACRO{\dsum \limits_{k=m}^{\infty}}%
		%BeginExpansion
		{\displaystyle\sum\limits_{k=m}^{\infty}}
		%EndExpansion
		q^{k}\left(  q^{m}-q^{k}\right) \\
		\nonumber&  =\frac{1-q^{2m}}{1-q^{2}}-q^{m}\frac{1-q^{m}}{1-q}+q^{m}\frac{q^{m}}%
		{1-q}-q^{2m}\frac{1}{1-q^{2}}\\
		&  =\frac{1}{1-q}\left(  \frac{1+2q^{2m+1}}{1+q}-q^{m}\right)
		\label{eq:sumaAbs}\end{align}
		we obtain the desired inequality \eqref{qOST}.
	\end{proof}

	\begin{remark}\label{rem:protuprimjer}
		Note that $q$-Ostrowski inequality in the previous theorem, assuming $q$-differentiability of $f$, can equivalently be written as
		\begin{multline}\label{eq:Nasa}
		\left|f(x)-\frac{1}{b-a}\int_a^bf(t)\,d^a_q t\right| \le (b-a)\left(\frac{1+2q\left(\frac{x-a}{b-a}\right)^2}{1+q}-\frac{x-a}{b-a}\right)\|D^a_qf\|_\infty
		\end{multline}
		where $x=a+q^m(b-a)$. But this is misleading since Theorem \ref{tm:OSTqNet} claims that inequality \eqref{eq:Nasa} holds only for $x\in[a,b]$ of the form $x=a+q^m(b-a)$.
		
		After some algebraic manipulations, it can be checked that inequality \eqref{eq:Nasa} is exactly the same as inequality \eqref{OST2} from \cite{TARI2}*{Theorem 3.5} where it is claimed that it holds for all $x\in[a,b]$.

		\medskip
		
		We stress again that inequality \eqref{OST2} (i.e.\ \eqref{eq:Nasa}) does not hold for all $x\in[a,b]$. It is not hard to find examples invalidating it. For example, set $a=0$, $b=1$, $q=\frac12$ and take function $f\colon [0,1]\to\R$ defined as
		$$f(x) = \begin{cases}
		x, &\text{ if } x \in [0,\frac9{10}],\\
		-9x+9, &\text{ if } x \in [\frac9{10},1].
		\end{cases}$$
		
		We leave it to the reader to confirm that $\di\int_{0}^{1} f(t) \, d_q t = \frac{q^2}{1+q} = \frac16$ and \linebreak $\|D_qf\|_\infty = 1$. Plugging all this into \eqref{OST2} (or \eqref{eq:Nasa}) for e.g.\ $x=\frac9{10}$ gives an obvious contradiction:
		$$\left|\frac9{10} - \frac16\right| \not\le \frac{23}{75}.$$
		Function $f$, as well as the bound it surpasses are shown in Figure \ref{fig:counterexample1}.
		
		\begin{figure}
			\begin{center}
				\includegraphics[trim = {1em 0 0 0 }]{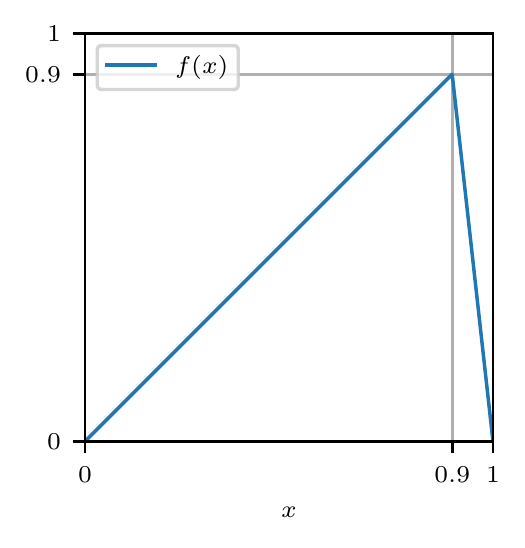}\hfil\includegraphics[trim = {1em 0 0 0 }]{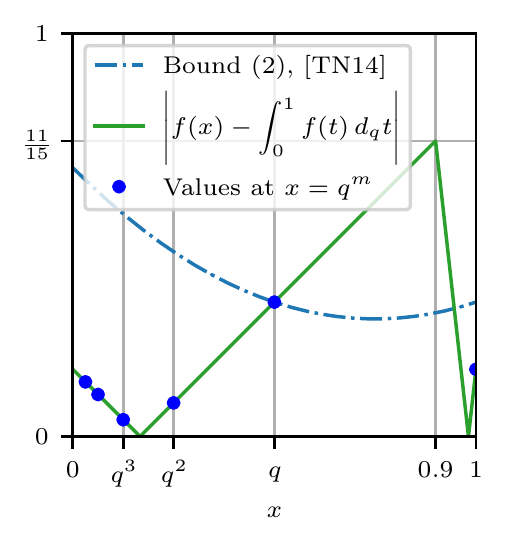}
			\end{center}
			
			\caption{A counterexample showing that inequality \eqref{OST2} does not hold for every $x\in[a,b]$. Note, however, that it does hold for every $x$ of the form $a+q^m(b-a)$ as is claimed in Theorem \ref{tm:OSTqNet}. (Here $q=\frac12$.)}\label{fig:counterexample1}
		\end{figure}
	\end{remark}
	
	\begin{theorem}
		Ostrowski inequality for $q$-calculus \eqref{qOST} is sharp for every
		$q\in\left\langle 0,1\right\rangle $.
	\end{theorem}
	
	\begin{proof}
		In order to simplify notation, we prove sharpness in case $a=0$ and $b=1$. Pre-composing the examples below with the affine transformation $t\mapsto \frac{t-a}{b-a}$ will produce examples that work for any $a$ and $b$.
		
		We will now show that for the
		function%
		\[
		f\left(  x\right)  =\left\vert x-q^{m}\right\vert \text{, }x\in\left[
		0,1\right]
		\]
		equality in \eqref{qOST} is obtained at $x=q^m$, i.e.
		\[
		\left\vert f\left(  q^{m}\right)  -%
		%TCIMACRO{\dint \limits_{0}^{1}}%
		%BeginExpansion
		{\displaystyle\int\limits_{0}^{1}}
		%EndExpansion
		f\left(  t\right) \, d_{q}t\right\vert =\left(  \frac{1+2q^{2m+1}}%
		{1+q}-q^{m}\right)  \left\Vert D_{q}f\right\Vert _{\infty}^{\langle 0,1]}%
		\]
		 We have $f\left(  q^{m}\right)  =0$ and
		$$
		{\displaystyle\int\limits_{0}^{1}}
		f\left(  t\right) \, d_{q}t =\left(  1-q\right)
		{\displaystyle\sum\limits_{k=0}^{\infty}}
		q^{k}f\left(  q^{k}\right) = (1-q)\sum_{k=0}^\infty q^k|q^k-q^m| =\frac{1+2q^{2m+1}}{1+q}-q^{m},
		$$
		where we used the previously computed sum \eqref{eq:sumaAbs}.
		
		Since it is obvious that%
		\[%
		\begin{array}
		[c]{ll}%
		D_{q}f\left(  t\right)  =-1, & 0\leq t\leq q^{m},\bigskip\\
		D_{q}f\left(  t\right)  \in\left\langle -1,1\right\rangle, & q^{m}<t<q^{m-1},\bigskip\\
		D_{q}f\left(  t\right)  =1, & q^{m-1}\leq t\leq1,\bigskip
		\end{array}
		\]
		we have $\left\Vert D_{q} f\right\Vert _{\infty}^{\left\langle 0,1\right]  }=1$
		and the equality holds.
	\end{proof}
	
	\begin{corollary}
		Let $f\colon [a,b]\to\R$ be a $q$-integrable function over $\left[  a,b\right]$. Then the
		following inequality holds%
		\[
		\left\vert f\left(  b\right)  -\frac{1}{b-a}%
		%TCIMACRO{\dint \limits_{a}^{b}}%
		%BeginExpansion
		{\displaystyle\int\limits_{a}^{b}}
		%EndExpansion
		f\left(  t\right)  d_{q}^{a}t\right\vert \leq\frac{q\left(  b-a\right)  }%
		{1+q}\left\Vert D_{q}^{a}f\right\Vert _{\infty}^{\langle a,b]}.
		\]
		
	\end{corollary}
	
	\begin{proof}
		Take $m=0$ in Theorem \ref{tm:OSTqNet}.
	\end{proof}
	
	\begin{corollary}\label{cor:mInfty}
		Let $f\colon [a,b]\to\R$ be a $q$-integrable function over $\left[  a,b\right]$. Additionally, assume that $f$ is continuous at $x=a$. Then the following inequality holds%
		\[
		\left\vert f\left(  a\right)  -\frac{1}{b-a}%
		%TCIMACRO{\dint \limits_{a}^{b}}%
		%BeginExpansion
		{\displaystyle\int\limits_{a}^{b}}
		%EndExpansion
		f\left(  t\right)  d_{q}^{a}t\right\vert \leq\frac{b-a}{1+q}\left\Vert
		D_{q}^{a}f\right\Vert _{\infty}^{\langle a,b]}.
		\]
		
	\end{corollary}
	
	\begin{proof}
		Let $m\rightarrow\infty$ in Theorem \ref{tm:OSTqNet}. The claim follows because $f$ is continuous at $x=a$, hence
		$\lim\limits_{m\to\infty} f\left(  a+q^{m}\left(  b-a\right)  \right)
		=f\left(  a\right)  $.
	\end{proof}

	\medskip
	
	The tight bound in the classical Ostrowski inequality is obtained for $x=\frac{a+b}{2}$.
	In that case \eqref{OST} reduces to midpoint inequality
	\[
	\left\vert f\left(\frac{a+b}{2}\right)-\frac{1}{b-a}\int_{a}^{b}f(t)dt\right\vert
	\leq\frac{b-a}{4}\left\Vert f^{\prime}\right\Vert _{\infty}.
	\]
	
	We now argue that for $q$-Ostrowski inequality the tight bound is obtained for $m=\left\lfloor \log_{q}\frac{1}{2}\right\rfloor$. For a fixed $q\in\left\langle 0,1\right\rangle $ we want to find
	\[
	\underset{m\in%
		%TCIMACRO{\U{2115} }%
		%BeginExpansion
		\mathbb{N}
		%EndExpansion
		\cup\left\{  0\right\}  }{\min}\left\{  \frac{1+2q^{2m+1}}{1+q}-q^{m}\right\}
	.
	\]

	It is easy to check that the function
	\[
	f\left(  x\right)  =\frac{1+2q^{2x+1}}{1+q}-q^{x}\text{, \ }x\in\left[
	0,\infty\right\rangle
	\]
	has only one critical point $x=\log_{q}\frac{1+q}{4q}$ at which a strict
	global minimum occurs. So we want to find the largest $m\in\N$ for which the following inequality holds%
	\[
	\frac{1+2q^{2m+1}}{1+q}-q^{m}\leq\frac{1+2q^{2\left(  m-1\right)  +1}}%
	{1+q}-q^{m-1}.
	\]
	In case this does not hold for any $m\in\N$, the minimum is clearly attained at $m=0$.
	
	From the inequality above we get
	\[
	q^{m-1}\left(1-2q^{m}\right)  \left(  1-q\right)  \leq0,
	\]%
	\[
	q^{m}\geq\frac{1}{2},
	\]
	hence
	\[
	m\leq\log_{q}\frac{1}{2},
	\]
	and therefore $m=\left\lfloor \log_{q}\frac{1}{2}\right\rfloor$.
	
	The following corollary is the
	midpoint inequality for $q$-calculus.
	
	\begin{corollary}[Midpoint inequality for $q$-calculus]
		Let $f\colon\left[
		a,b\right]  \rightarrow%
		%TCIMACRO{\U{211d} }%
		%BeginExpansion
		\mathbb{R}
		%EndExpansion
		$ be a $q$-integrable function over $\left[  a,b\right]  $. Then the following
		inequality holds%
		\begin{align*}
		&  \left\vert f\left(  a+q^{\left\lfloor \log_{q}\frac{1}{2}\right\rfloor
		}\left(  b-a\right)  \right)  -\frac{1}{b-a}%
		%TCIMACRO{\dint \limits_{a}^{b}}%
		%BeginExpansion
		{\displaystyle\int\limits_{a}^{b}}
		%EndExpansion
		f\left(  t\right)  d_{q}^{a}t\right\vert \\
		&  \leq\left(  b-a\right)  \left(  \frac{1+2q^{2\left\lfloor \log_{q}\frac
				{1}{2}\right\rfloor +1}}{1+q}-q^{\left\lfloor \log_{q}\frac{1}{2}\right\rfloor
		}\right)  \left\Vert D_{q}^{a}f\right\Vert _{\infty}^{\left\langle a,b\right]
		}.
		\end{align*}
		
	\end{corollary}
	
	\begin{proof}
		Take $m=\left\lfloor \log_{q}\frac{1}{2}\right\rfloor$ in Theorem \ref{tm:OSTqNet}.
	\end{proof}
	
	\section{\texorpdfstring{$q$}{q}-Ostrowski inequality}\label{sec:OstrowskiFull}
	
	We shall now derive the correct bound for all the other $x\in[a,b]$ which are not of the form $x=a + q^m(b-a)$.
	
	\begin{theorem}\label{tm:our_ostrowski}
		Let $f\colon [a,b] \to \R$ be $q$-integrable over $[a,b]$, and further assume that $f$ is continuous at $x=a$. Then for all $x\in[a,b]$ the following inequality holds
		\begin{equation}\label{ineq:ostrowski_ineq}
		\left|f(x)-\frac{1}{b-a}\int_a^bf(t)\, d^a_qt\right| \le
		(b-a)\left(\frac{x-a}{b-a}+\frac{1}{1+q}\right)\|D^a_qf\|_\infty^{\langle a,b]}.
		\end{equation}
		Moreover, this bound is sharp whenever $x$ is not of the form $a+q^m(b-a)$.
	\end{theorem}
	
	\begin{remark}
		Note that the bound given in this theorem is strictly worse than the one in Theorem \ref{tm:OSTqNet}. But unlike that bound, this one holds for all $x\in[a,b]$.
	\end{remark}
	
	\begin{proof}[Proof of Theorem \ref{tm:our_ostrowski}]
		In Theorem \ref{thm:MVT} we proved a version of mean value inequality for $q$-calculus:
		$$\left|f(x)-f(a+q^n(x-a))\right|\le |x-(a+q^n(x-a)| \, \|D^a_qf\|_{\infty}^{\langle a,b]}, $$
		holds for all $x\in[a,b]$, and $n \in \N\cup\{0\}$. Fixing $x$ and letting $n\to\infty$ we get
		$$\left|f(x)-f(a)\right|\le (x-a)\, \|D^a_qf\|_{\infty}^{\langle a,b]}, \text{ for all } x\in[a,b],$$
		where we made use of the continuity of $f$ at $x=a$.
		
		From Corollary \ref{cor:mInfty} we know that
		$$\left|f(a)-\frac{1}{b-a}\int_a^bf(t)\,d^a_q t\right| \le \frac{b-a}{1+q}\|D^a_qf\|_\infty^{\langle a,b]}.$$
		Putting these two inequalities together we obtain the claim
		\begin{multline*}
		\left|f(x)-\frac{1}{b-a}\int_a^bf(t)\,d^a_q t\right| \le
		\left|f(x)-f(a)\right| + \left|f(a)-\frac{1}{b-a}\int_a^bf(t)\,d^a_q t\right| \\
		\le \left(x-a+\frac{b-a}{1+q}\right)\|D^a_qf\|_\infty^{\langle a,b]} = (b-a)\left(\frac{x-a}{b-a}+\frac{1}{1+q}\right)\|D^a_qf\|_\infty^{\langle a,b]}.
		\end{multline*}
	\end{proof}
	
	\subsection{Sharpness of the inequality}
	In order to simplify notation in this subsection, we set $a=0$ and $b=1$. All the claims hold in full generality for any shifted domain $[a,b]$ and corresponding $q$-derivative $D_q^a$ and $q$-integral.

	Our goal is to show that for any $q\in\langle 0,1 \rangle$, and for each $x\in[0,1]$ not of the form $x=q^n$, there exist a function $f_x\colon[0,1]\to\R$ which attains the bound
	\begin{equation}\label{eq:bound01}
	\left|f(x)-\int_0^1 f(t) \, d_qt\right| \le
	\left(x+\frac{1}{1+q}\right)\|D_qf\|_\infty^{\langle 0,1]}.
	\end{equation}
	We assume that $q\in\langle 0,1 \rangle$, once it was chosen, is fixed.
	
	Observe that inequality \eqref{eq:bound01} is scale invariant, so it is sufficient to look for examples $f$ with $\|D_qf\|_\infty^{\langle 0,1]}=1$. Setting $\tilde{f} = Mf$ for any $M \ge 0$ then produces examples with $\|D_q \tilde{f}\|_\infty^{\langle 0,1]}=M$ which also attain the bound.
	
	We now proceed with the construction of $f_x$ where we fixed $x\in[0,1]$ not of the form $x=q^n$. We further assume that $x\in\langle q,1\rangle$. We will show how to treat other $x$ afterwards.
	
	Let $f_x\colon[0,1]\to\R$ be a function such that $f_x(q^n) = -q^n$ and $f_x(xq^n) = xq^n$ for all $n\in\N\cup\{0\}$. Further let $f_x(0)=0$, and for all other $t\in[0,1]$ let $f_x(t)$ be defined as the linear interpolation of the previously defined points. More precisely, let $f_x$ be defined as:
	$$f_x(t) = \begin{cases}
	-\frac{1+x}{1-x}t+\frac{2x}{1-x}q^n, &\text{if } xq^n \le t\le q^n, \\
	\frac{x+q}{x-q}t-\frac{2qx}{x-q}q^n, &\text{if } q^{n+1}\le t\le xq^n,
	\end{cases}$$
	where $n\in\N\cup\{0\}$ and additionally $f_x(0)=0$ see Figure \ref{fig:nondifexample}.
	
	\begin{figure}
		\begin{center}
			\includegraphics[trim = {2em 0 0 0 }]{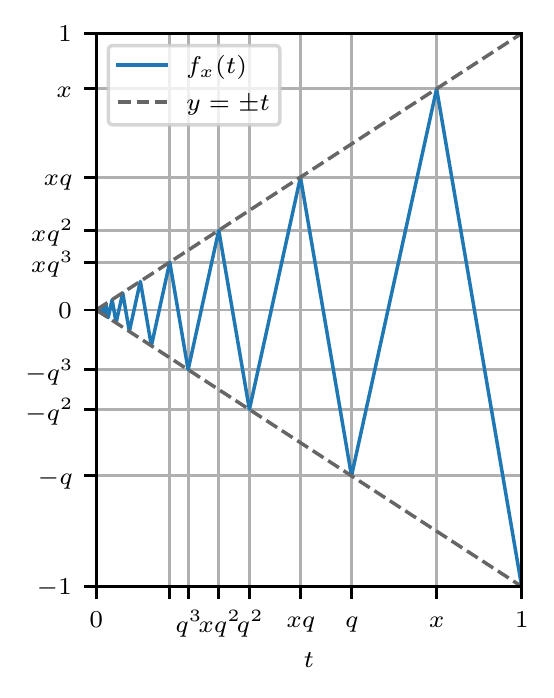}\hfil\includegraphics[trim = {2em 0 0 0 }]{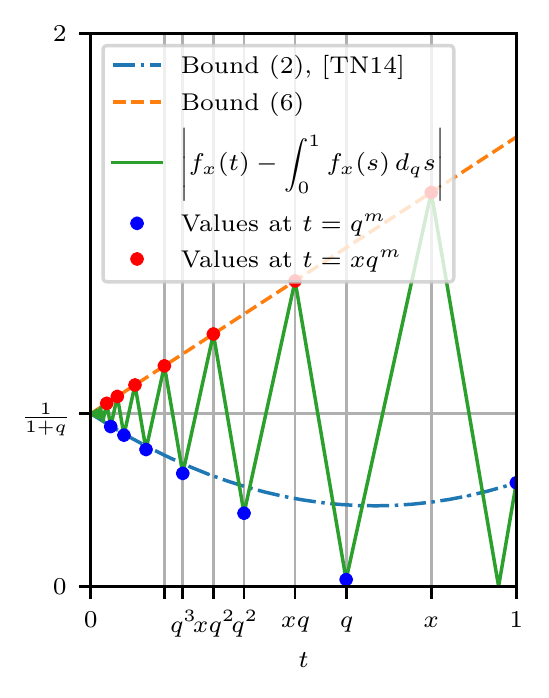}
		\end{center}
		\caption{An example showing that the bound in Theorem \ref{tm:our_ostrowski} is sharp. Here $q=0.6$, $x=0.8$, and $f_x$ attains the bound at all $t=xq^m$. Note that $f_x$ is not $q$-differentiable at $t=0$, but somewhat surprisingly, its $q$-derivative over $\langle0,1]$ never exceeds 1 in absolute value. Also, the incorrect bound \eqref{OST2} does hold for $t$ of the form $t=q^m$.}\label{fig:nondifexample}
	\end{figure}
	
	Note that $f_x$ is continuous on $[0,1]$ but is not $q$-differentiable at $t=0$. Later, we will show how to alter these examples to produce $q$-differentiable examples.
	
	\begin{proof}[Claim 1] $f_x$ is $q$-differentiable on $\langle 0,1]$ with $\|D_qf\|_\infty^{\langle 0,1]}=1$.
		
		Note that $f_x$ is self-similar. To be precise, $f_x$ over $\langle q^2,q^1]$ is scaled version of $f_x$ over  $\langle q,1]$. Indeed, it is readily checked that $f_x(qt) = qf_x(t)$ for all $t\in\langle0,1]$. This allows us to easily compute $D_qf_x$.
		$$D_qf_x(t) = \frac{f_x(t)-f_x(qt)}{t(1-q)}=\frac{f_x(t)-qf_x(t)}{t(1-q)} = \frac{f_x(t)}{t},$$
		hence
		$$|D_qf_x(t)| = \left|\frac{f_x(t)}{t}\right| \le 1.$$
	\end{proof}
	
	\begin{proof}[Claim 2] $f_x$ is $q$-integrable on $[0,1]$ with $\di\int_{0}^{1}f_x(t) \, d_qt = -\frac{1}{1+q} $.
		
		We calculate:
		\begin{multline*}
		\int_{0}^{1}f_x(t) \, d_qt = (1-q)\sum_{n=0}^{\infty} f(q^n)q^n = (1-q)\sum_{n=0}^{\infty} -q^nq^n = \\
		= (1-q)\sum_{n=0}^{\infty} -q^{2n} = -\frac{1}{1+q}
		\end{multline*}
	\end{proof}
	
	Lastly, note that
	$$\left|f_x(x) - \int_{0}^{1}f_x(t) \, d_qt\right| = \left|x - \left( -\frac{1}{1+q}\right) \right| = x + \frac{1}{1+q}$$
	which completes the proof that $f_x$ attains the bound at $x$ for $x\in\langle q,1 \rangle$.
	
	For $x \not\in \langle q,1 \rangle$, we first find $n\in\N$ such that $q^{n+1} < x < q^n$ and set $\tilde{x} = \dfrac{x}{q^n}$. Now $\tilde{x} \in \langle q,1 \rangle$ and for function $f_{\tilde{x}}$ as above we have
	$$\left|f_{\tilde{x}}(x) - \int_{0}^{1}f_{\tilde{x}}(t) \, d_qt\right| = \left|f_{\tilde{x}}(\tilde{x}q^n) - \left( -\frac{1}{1+q}\right) \right| = \tilde{x}q^n + \frac{1}{1+q} = x+\frac{1}{1+q}$$
	so $f_{\tilde{x}}$ attains the bound at $x$ for $x\in\langle q^{n+1}, q^n \rangle$.
	
	\medskip
	
	Putting together Theorems \ref{tm:OSTqNet} and \ref{tm:our_ostrowski} and taking into account the examples following each, we obtain the following corollary.
	\begin{corollary}[Full $q$-Ostrowski inequality]
		Let $f\colon [a,b] \to \R$ be $q$-integrable over $[a,b]$, and further assume that $f$ is continuous at $x=a$. Then for all $x\in[a,b]$ the following sharp inequality holds
		\begin{equation*}
		\left|f(x)-\frac{1}{b-a}\int_a^bf(t)\, d^a_qt\right| \le M(x)\|D^a_qf\|_\infty^{\langle a,b]},
		\end{equation*}
		where $M(x)$ is a discontinuous function:
		$$ M(x) = 
		\begin{cases}
		(b-a)\left(\frac{1+2q\left(\frac{x-a}{b-a}\right)^2}{1+q}-\frac{x-a}{b-a}\right), & \text{ if } x=a+q^m(b-a) \text{ for } m\in\N\cup\{0\},\\
		(b-a)\left(\frac{x-a}{b-a}+\frac{1}{1+q}\right), & \text{ otherwise.}
		\end{cases}
		$$
	\end{corollary}
	
	\begin{figure}
		\includegraphics[trim={4.5em 0 0 0}]{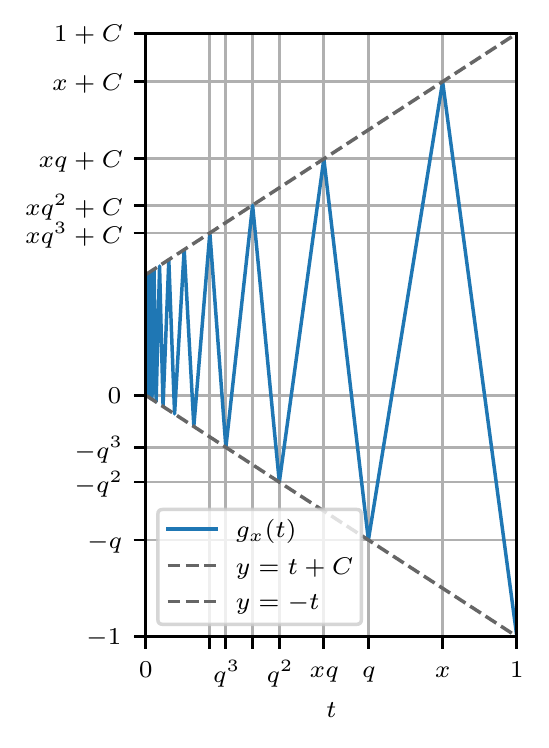}
		\caption{}\label{fig:nonlim}
	\end{figure}

	\begin{remark}
		It is reasonable to ask whether the hypothesis in Theorem \ref{tm:our_ostrowski}, and consequently in the corollary above, could be further relaxed.
		
		It is immediately clear that instead of requiring $f$ to be continuous at $x=a$, it is sufficient to ask for existence of limit $\lim_{x\to a+}f(x)$ but then the claim only holds for $x\in\langle a,b]$. If that is the case, one might as well redefine $f$ at $x=a$ in order to make it continuous.
		
		Relaxing the hypothesis even further is not possible as the following example shows. Assuming $a=0$, $b=1$, construct a piecewise linear map $g_x\colon \langle 0,1] \to \R$ in a similar fashion as $f_x$ was constructed before, but set $g_x(q^n) = -q^n$ and $g_x(xq^n) = xq^n+C$ for some fixed $C>0$, and all $n\in\N\cup\{0\}$, see Figure \ref{fig:nonlim}. Note that $\lim_{x\to 0+}f(x)$ does not exist. Now, using Lemma \ref{lm:derivative_inbetween} below, it can be checked that $|D_q g_x(t)|\le 1$ for $t\in\langle 0,1]$, and $\left|f(x)-\int_0^1f(t)\, d_q t\right|$ can be made arbitrarily large by choosing a large $C$.
	\end{remark}
	
	\subsection{\texorpdfstring{$q$}{q}-differentiable examples}
	One might wonder whether imposing more restrictions on $f$ in Theorem \ref{tm:our_ostrowski} could lead to a better bound. We shall show that asking for $q$-differentiability (at $a$) does not improve things. We do this by constructing $q$-differ\-entiable examples that show the bound in Theorem \ref{tm:our_ostrowski} is best possible even if one considers a class of $q$-differentiable functions. The examples we shall construct will be piecewise linear functions with finitely many pieces. The following lemma will make it easier to bound the $q$-derivative of such a function.
	
	\begin{lemma}\label{lm:derivative_inbetween}
		Let $c,d\in\R$ be such that $0<qc<qd<c<d\le 1$, and assume that $f\colon [0,1] \to \R$ is a function whose values over $[qc,qd]$ are obtained as linear interpolation of values of $f$ at endpoints $qc$ and $qd$; and similarly the values over $[c,d]$ are linear interpolation of values of $f$ at endpoints $c$ and $d$ (see Figure \ref{fig:difspline}).
		
		More precisely, assume that $f(\alpha c + (1-\alpha) d) = \alpha f(c) + (1-\alpha)f(d)$ and $f(\alpha qc + (1-\alpha)qd) = \alpha f(qc) + (1-\alpha)f(qd)$ for all $\alpha\in[0,1]$.
		
		Then, the $q$-derivative of $f$ over $[c,d]$ satisfies the following inequality
		$$\min(D_qf(c),D_qf(d)) \le D_q f(t) \le \max(D_qf(c),D_qf(d))\,,\quad \text{ for all } t\in[c,d].$$
	\end{lemma}
	
	\begin{figure}
		\includegraphics[trim = {2em 0 0 0 }]{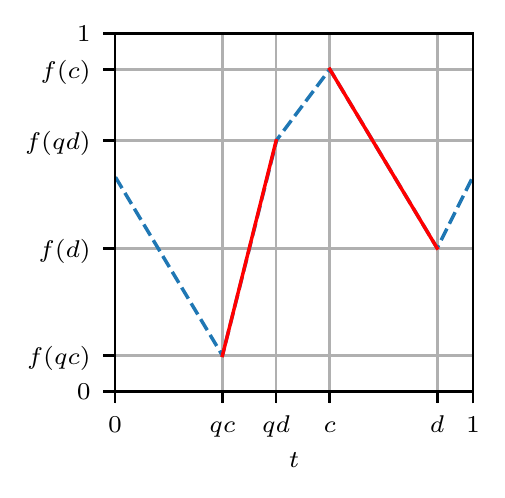}
		\caption{An example of a function for which Lemma \ref{lm:derivative_inbetween} applies. Note that the values of $f$ corresponding to dashed lines are irrelevant and $f$ is not required to be piecewise linear there.}\label{fig:difspline}
	\end{figure}
	
	\begin{proof}
		We first write $t\in[c,d]$ as $t=\alpha c+ (1-\alpha) d$ for some $\alpha\in[0,1]$. We then calculate
		\begin{align*}
		D_qf(t) &= \frac{f(t)-f(qt)}{(1-q)t} = \frac{f(\alpha c+ (1-\alpha) d) - f(\alpha qc+ (1-\alpha) qd)}{(1-q)(\alpha c+ (1-\alpha) d)} \\
		&=\frac{\alpha f(c) + (1-\alpha)f(d) - \alpha f(qc) - (1-\alpha)f(qd)}{(1-q)(\alpha c+ (1-\alpha) d)} \\
		&=\frac{\alpha (f(c) - f(qc)) + (1-\alpha)(f(d) -f(qd))}{(1-q)(\alpha c+ (1-\alpha) d)} \\
		&=\frac{\alpha c D_qf(c) + (1-\alpha) d D_qf(d)}{(\alpha c+ (1-\alpha) d)}
		\end{align*}
		Since the last expression is a weighted average of $D_qf(c)$ and $D_qf(d)$ the statement of the lemma follows.
	\end{proof}
	
	We will now show that for any $x$ not of the form $x=q^m$, and any $\epsilon>0$ we can find a $q$-differentiable function $f_{x,\epsilon}$ which comes $\epsilon$-close to the bound in Theorem \ref{tm:our_ostrowski}, thus showing that the bound in the theorem is best possible.
	
	As before, we shall first show that we can do this for $x\in\langle q, 1 \rangle$. Once $x$ is fixed, let $\epsilon > 0$ be sufficiently small ($\epsilon\le 2x$), and let $f_{x,\epsilon} \colon [0,1] \to \R$ be a function such that $f_{x,\epsilon}(q^n) = -q^n$, and $f_{x,\epsilon}(xq^n) = \max(xq^n-\epsilon,-xq^n)$ for all $n\in\N\cup\{0\}$. Note that $f_{x,\epsilon}(x) = x-\epsilon$.
	
	If we denote by $m\in\N\cup\{0\}$ the integer such that $xq^{m+1} < \frac{\epsilon}{2}\le xq^m$, i.e.\ $m=\lfloor\log_q\frac{\epsilon}{2x}\rfloor$, then $f_{x,\epsilon}(xq^n) = xq^n-\epsilon$ for all $n\le m$ and $f_{x,\epsilon}(xq^n) = -xq^n$ for all integers $n>m$.
	
	Further let $f_{x,\epsilon}(0) = 0$, and for all other $t\in[0,1]$ let $f_{x,\epsilon} (t)$ be defined as the linear interpolation of the previously defined points. Graph of $f_{x,\epsilon}(xq^n)$ is shown in Figure \ref{fig:difexample}.
	
	\begin{figure}
		\begin{center}
			\includegraphics[trim = {3.2em 0 0 0 }]{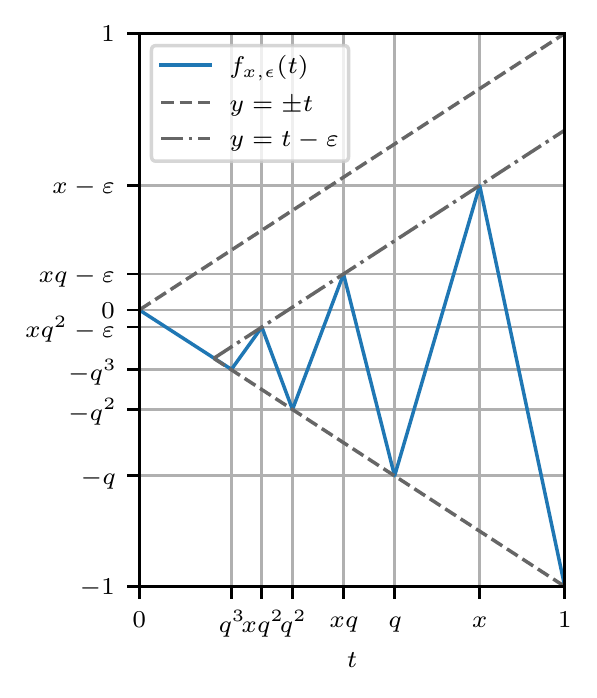}\hfil\includegraphics[trim = {2em 0 0 0 }]{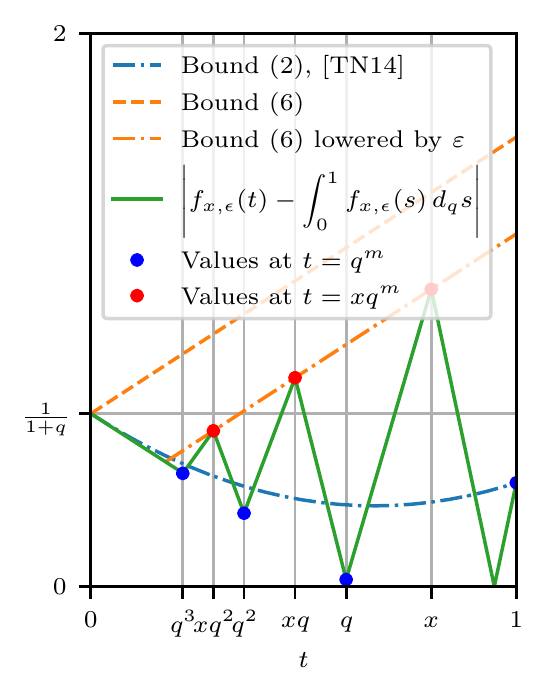}
		\end{center}
		\caption{A $q$-differentiable example showing that the bound in Theorem \ref{tm:our_ostrowski} is best possible. Here $q=0.6$, $x=0.8$, and $f_{x,\epsilon}$ comes $\epsilon$-close to the bound at all $t=xq^m$. Again, the incorrect bound \eqref{OST2} does hold for $t$ of the form $t=q^m$.}\label{fig:difexample}
	\end{figure}
	
	\begin{proof}[Claim 1] $f_{x,\epsilon}$ is continuously $q$-differentiable over $[0,1]$.
		
		Note that $f_{x,\epsilon}(t) = -t$ for all $t\in[0,q^{m+1}]$, and $f_{x,\epsilon}$ is therefore differentiable at $t=0$. This further means that $f_{x,\epsilon}$ is continuously $q$-differentiable over $[0,1]$ as $f_{x,\epsilon}$ itself is continuous.
	\end{proof}
	
	\begin{proof}[Claim 2] $f_{x,\epsilon}$ is $q$-integrable on $[0,1]$ with $\di\int_{0}^{1}f_{x,\epsilon}(t) \, d_qt  = -\frac{1}{1+q}$.
		
		The values $f_{x,\epsilon}$ match those of $f_x$ at $t=q^n$ for all $n\in\N\cup\{0\}$. Hence,
		$$\int_{0}^{1}f_{x,\epsilon}(t) \, d_qt  = \int_{0}^{1}f_x(t) \, d_qt = -\frac{1}{1+q}.$$
	\end{proof}
	
	\begin{proof}[Claim 3] $|D_q f_{x,\epsilon}(t)| \le 1$ for all $t\in[0,1]$.
		
		To show this we employ Lemma \ref{lm:derivative_inbetween}. It is follows immediately from the definition of $f_{x,\epsilon}$ that $D_q f_{x,\epsilon}(q^n) = -1$ for all $n\in\N\cup \{0\}$. Also, $D_q f_{x,\epsilon}(xq^n) = 1$ for all $n\le m-1$, and $D_q f_{x,\epsilon}(xq^n) = -1$ for $n \ge m+1$. The only problematic point is $t=xq^m$ so we calculate
		\begin{align*}
		D_q f_{x,\epsilon}(xq^{m}) &= \frac{f_{x,\epsilon}(xq^{m}) - f_{x,\epsilon}(xq^{m+1})}{(1-q)xq^{m}} = \frac{xq^{m}-\epsilon-(-xq^{m+1})}{(1-q)xq^{m}} = \frac{1+q-\frac{\epsilon}{xq^{m}}}{1-q}.
		\end{align*}
		
		Taking into account that $xq^{m+1} < \frac{\epsilon}{2}\le xq^m$, it follows that $2q <\frac{\epsilon}{xq^{m}} \le 2$, and hence $D_q  f_{x,\epsilon}(xq^{m}) \in [-1,1]$.
		
		To sum up, we have shown that $D_q  f_{x,\epsilon} (t) \in [-1,1]$ for all $t$ that are equal to $q^n$ or $xq^n$ for some $n\in\N\cup\{0\}$. Since any other $t\in \langle 0,1 \rangle$ lies inside some interval $[c=xq^n, d=q^n]$ or inside $[c=q^{n+1}, d=xq^n]$ for some $n$, and since the endpoints of interval $[qc,qd]$ are again two neighbouring points of the same form, and since $f_{x,\epsilon}$ is defined to be linear interpolation of values at those endpoints --- all the conditions of Lemma \ref{lm:derivative_inbetween} are met. We can therefore conclude that $\left| D_q  f_{x,\epsilon} (t)\right|\le 1$ for all $t\in[0,1]$.
	\end{proof}
	
	It now only remains to note that
	$$\left|f_{x,\epsilon}(x) - \int_{0}^{1}f_{x,\epsilon}(t) \, d_qt\right| = \left|x - \epsilon - \left( -\frac{1}{1+q}\right) \right| = x + \frac{1}{1+q} - \epsilon$$
	which completes the proof that there are $q$-differentiable functions $f_{x,\epsilon}$ that come arbitrarily close to the bound at $x$ for any $x\in\langle q,1 \rangle$.
	
	As before, if $x \not\in \langle q,1 \rangle$, we first find $n\in\N$ such that $q^{n+1} < x < q^n$ and set $\tilde{x} = \dfrac{x}{q^n}$ so that $\tilde{x} \in \langle q,1 \rangle$. Then choose $\epsilon>0$ sufficiently small so that $\max(\tilde{x}q^n-\epsilon,-\tilde{x}q^n) = \tilde{x}q^n-\epsilon$, i.e.\ $\epsilon\le 2x$. Now for the function $f_{\tilde{x},\epsilon}$ as above we have
	\begin{multline*}
	\left|f_{\tilde{x},\epsilon}(x) - \int_{0}^{1}f_{\tilde{x},\epsilon}(t) \, d_qt\right| = \left|f_{\tilde{x},\epsilon}(\tilde{x}q^n) - \left( -\frac{1}{1+q}\right) \right| = \tilde{x}q^n-\epsilon+ \frac{1}{1+q} = \\
	= x+\frac{1}{1+q}-\epsilon
	\end{multline*}
	so $f_{\tilde{x},\epsilon}$ comes arbitrarily close to the bound at $x$.
	
	\begin{remark}
		In the light of the previous examples, one could ask whether there are any $q$-differentiable functions that achieve the bound \eqref{ineq:ostrowski_ineq} exactly at some $x$ not of the form $x=a+q^m(b-a)$. We do not know of such examples and conjecture that any such an example will fail to be $q$-differentiable at $a$.
	\end{remark}

	\section*{Acknowledgments}
	Domagoj Kovačević was supported by the QuantiXLie Centre of Excellence, a project
	co financed by the Croatian Government and European Union through the
	European Regional Development Fund - the Competitiveness and Cohesion
	Operational Programme (Grant KK.01.1.1.01.0004).
	

\end{document}